\documentclass[a4paper,11pt,reqno]{amsart} 
\usepackage{a4,amsbsy,amscd,amssymb,amstext,amsmath,latexsym,oldgerm,enumerate,verbatim,graphicx}
\usepackage[all]{xy}

\newtheorem{thmgl} {Theorem}    
\newtheorem{propgl}{Proposition}
\newtheorem{lemgl} {Lemma}

\newtheorem{cornn}{Corollary}

\newtheorem{connn} {Conjecture}

\theoremstyle{definition}

\newtheorem{rem} {Remark} [section]
\newtheorem{rems} [rem]{Remarks}

\newcommand{\mf}{\mathfrak}
\newcommand{\mc}{\mathcal}

\newcommand{\ov}{\overline}

\newcommand{\sm}{\setminus}         

\newcommand{\ot}{\otimes}           
\newcommand{\la}{\langle}
\newcommand{\ra}{\rangle}

\newcommand{\ind}{{\rm ind}}

\newcommand{\Ker}{{\rm Ker}}

\newcommand{\g}{\mf{g}}

\newcommand{\p}{\mf{p}}
\newcommand{\q}{\mf{q}}

\newcommand{\fl}{\mf{l}}

\newcommand{\fu}{\mf{u}} 

\newcommand{\Lie}{{\rm Lie}}

\newcommand{\GL}{{\rm GL}}

\newcommand{\SL}{{\rm SL}}

\newcommand{\e}{\varepsilon}

\begin{document}

\title{A Frobenius splitting and cohomology vanishing for the cotangent bundles of the flag varieties of $\GL_n$}

\begin{abstract}
Let $k$ be an algebraically closed field of characteristic $p>0$, let $G=\GL_n$ be the general linear group over $k$, let $P$ be a parabolic subgroup of $G$, and let $\fu_P$ be the Lie algebra of its unipotent radical.
We show that the Kumar-Lauritzen-Thomsen splitting of the cotangent bundle $G\times^P\fu_P$ of $G/P$ has top degree $(p-1)\dim(G/P)$. The component of that degree is therefore given by the $(p-1)$-th power of a function $f$. We give a formula for $f$ and deduce that it vanishes on the exceptional locus of the resolution $G\times^P\fu_P\to\ov{\mc O}$ where $\ov{\mc O}$ is the closure of the Richardson orbit of $P$. As a consequence we obtain that the higher cohomology groups of a line bundle on $G\times^P\fu_P$ associated to a dominant weight are zero. The splitting of $G\times^P\fu_P$ given by $f^{p-1}$ can be seen as a generalisation of the Mehta-Van der Kallen splitting of $G\times^B\fu$.
\end{abstract}

\author[R.~Tange]{Rudolf Tange}
\address
{School of Mathematics,
University of Leeds,
LS2 9JT, Leeds, UK}
\email{R.H.Tange@leeds.ac.uk}

\maketitle
\markright{\MakeUppercase{Frobenius splitting and Cohomology vanishing for $T^\vee(\GL_n/P)$}}

\section*{Introduction}
Let $G$ be a reductive group over an algebraically closed field $k$ of positive characteristic $p$. For a parabolic $P$ containing the positive Borel and $P$-module $M$, we denote by $H^i(G/P,M)$ the $i$-th chomology group of the sheaf $\mc L_{G/P}(M)$ on $G/P$ associated to $M$. It is an open problem whether we have for all parabolic subgroups $P$ and all dominant characters $\lambda$ of $P$ that
$$H^i(G/P,S(\fu_P^*)\ot k_{-\lambda})=0\text{\quad for all\ }i>0\,,\eqno{(*)}$$
where the most important case is $\lambda=0$, see e.g. \cite[Introduction to Ch~5]{BriKu}.
In characteristic 0 this is an easy consequence of the Grauert-Riemenschneider Theorem, see \cite[Thm~2.2]{Broer2}.
In characteristic $p$ (*) is known for $P=B$, for arbitrary $P$ and ``$P$-regular" dominant $\lambda$, see \cite{KLT}, and for $P$ corresponding to sets of pairwise orthogonal short simple roots and $\lambda=0$, see \cite{Th}.

It is easy to write a formula for the Euler character $$\sum_{i\ge0}(-1)^i{\rm ch}\,H^i(G/P,S(\fu_P^*)\ot k_{-\lambda})\,,$$
see \cite[Sect~8.14-8.16]{Jan1} and \cite[Prop~2.1]{Broer1}, so if (8) holds we get a formula for ${\rm ch}\,H^0(G/P,S(\fu^*)\ot k_{-\lambda})$.

When $\mc L_{G/P}(\lambda)=\mc L_{G/P}(k_{-\lambda})$ is ample, i.e. $\lambda$ ``$P$-regular" dominant, one gets (*) from the fact that $G\times^P\fu_P$ is Frobenius split. One can also use Frobenius splittings to prove (*) for $\lambda=0$ via a characteristic $p$-version of of the Grauert-Riemenschneider Theorem \cite[Thm~1.2]{MVdK0}, since the canonical bundle of $G\times^P\fu_P$ is trivial. But then the map from $G\times^P\fu_P$ to the Richardson orbit closure has to be birational and the splitting has to be a $(p-1)$-th power of a section $\sigma$ of the anti-canonical bundle which vanishes on the excptional locus. This is the approach we will follow.

When I asked Thomsen about the case $G=\GL_n$, he told me he expected that the pushforward to $G\times^P\fu_P$ of the splitting of $G\times^B\fu_P$ induced by the ``MVdK-splitting" of $G\times^B\fu$ from \cite{MVdK} is the homogeneous component of degree\break $(p-1)\dim(G/P)$ of the ``KLT-splitting", see Section~\ref{s.splitting}, from \cite{KLT}. Although we can not prove this conjecture, we can show that the above component is in fact the top degree component and therefore a $(p-1)$-th power. From this we can then deduce that this homogeneous splitting vanishes on the exceptional locus of the resolution $\varphi:[g,X]\mapsto gXg^{-1}:G\times^P\fu_P\to\ov{\mc O}$, where $\ov{\mc O}$ is the closure of the Richardson orbit corresponding to $P$, see Theorem~\ref{thm.splitting} in Section~\ref{s.main}. Finally, we then deduce that (*) holds in type $A$, see Theorem~\ref{thm.cohomology}. In fact we can formulate this as a result for arbitrary reductive groups.

The main idea of the proof is as follows. The ``KLT-splitting" from \cite{KLT} is the $(p-1)$-th power of the pullback along $\varphi$ of the function which maps an $n\times n$ matrix $X$ to
\begin{align}\label{eq.splitting}
\prod_{i=1}^{n-1}\det\big((I_n+X)_{\le i,\le i}\big)\,,
\end{align}
where $Y_{\le i,\le i}$ denotes the submatrix of $Y$ given by the first $i$ rows and columns, see \cite[Example~5.1.15]{BriKu}.\footnote{Apart from the degree computation, the arguments there work for any parabolic.}
Unlike in the case $P=B$, the degree of the $i$-th factor may be less than $i$. In Lemma's~\ref{lem.upperbound}(ii)~and~\ref{lem.existence} we determine the degree of the $i$-th factor and from that it follows that the product \eqref{eq.splitting} has degree $\dim(G/P)$.

\section{Preliminaries}\label{s.prelim}
\subsection{Notation}
Let $k$ be an algebraically closed field of characteristic $p>0$ and let $G$ be a reductive group over $k$. We fix a Borel subgroup $B\le G$ and maximal torus $T\le B$. We denote by $R$ the set of roots of $T$ in the Lie algebra $\g=\Lie(G)$ of $G$, and we denote the unipotent radical of $B$ by $U$. We call the roots of $T$ in $\fu=\Lie(U)$ positive and we denote the corresponding set of simple roots by $S$. For a subset $I$ of $S$ we denote the root system spanned by $I$ by $R_I$. Furthermore, we denote the corresponding parabolic subgroup containing $B$ and its Levi subgroup containing $T$ by $P_I$ and $L_I$. Denote the character group of an algebraic group $H$ by $X(H)$. For $I\subseteq S$ we identify $X(P_I)$ and $X(L_I)$ with $\{\lambda\in X(T)\,|\,\la\lambda,\alpha^\vee\ra=0\text{\ for all\ }\alpha\in I\}$.

For $P$ a parabolic of $G$ and $M$ a $P$-module we write $\mc L(M)$ for the $G$-linearised sheaf on $G/P$ associated to $M$.
For $\lambda\in X(P)\le X(T)$ we put $\mc L(\lambda)=\mc L(k_{-\lambda})$, it is the sheaf of sections of the line bundle $G\times^Pk_{-\lambda}$ on $G/P$. We use the same symbol $\mc L(\lambda)$ to denote the sheaf of sections of the pullback of this line bundle to $G\times^PV$ for any $P$-variety $V$. We also write $H^i(G/P,M)$ for
$$H^i(G/P,\mc L(M))\simeq R^i\ind_P^G(M)\,,$$ see \cite[I.5.12]{Jan}.
We have that $$H^i(G\times^P\fu_P,\mc L(\lambda))=H^i(G/P,k[\fu_P]\ot k_{-\lambda})\,,$$
see \cite[Lem~5.2.2]{BriKu}.

If $p={\rm char}\,k$ is good for $G$, then we have $\g/\p\simeq\fu_P^*$ as $P$-modules and $G\times^P\fu_P$ is the cotangent bundle $T^\vee(G/P)$ of $G/P$, see \cite[5.1.8-11]{BriKu}.

\subsection{Frobenius splittings}\label{s.splitting}
By \cite[Lem~5.1.1]{BriKu} the canonical bundle of $G\times^P\fu_P$ is trivial, so we can choose a nowhere zero global section: a volume form. It is easy to see that such a section is unique up to a scalar multiple, see \cite[5.1.2]{BriKu}. This means that we can think of Frobenius splittings (up to a scalar multiple) of $G\times^P\fu_P$ as certain regular functions on $G\times^P\fu_P$.

In \cite[Thm~1]{KLT} it was proved that, when $p$ is good for $G$, the cotangent bundle $T^\vee(G/P)$ of $G/P$ is Frobenius split, see also \cite[Thm 5.1.3]{BriKu}. We will refer to the $B$-canonical splitting $\psi_P(f_-\ot f_+)$ as the ``KLT-splitting" of $T^\vee(G/P)$, where $\psi_P,f_-,f_+$ are as defined in \cite[Ch~5]{BriKu}. Actually this is only a splitting up to a scalar multiple, but in the case $G=\GL_n$ we assume that the chosen volume form on $T^\vee(G/P)$ is such that the pullback along $\varphi$ of the function given by \eqref{eq.splitting}, $\varphi$ defined as in the introduction, defines a splitting. That formula is all we need to know about the KLT-splitting in this paper.

The standard grading of $k[\fu_P]=S(\fu_P^*)$ gives a grading on $k[G\times^P\fu_P]$, and in \cite[5.1.14]{BriKu} it is explained that the homogeneous component of degree $(p-1)\dim(G/P)$ of a splitting $\sigma$ of $G\times^P\fu_P$ is again a splitting of $G\times^P\fu_P$. This component is $B$-canonical if $\sigma$ is $B$-canonical.

\subsection{A result on cohomology vanishing}
The following result is probably well-known, see e.g. \cite[Sect~7.2]{BNPP}, but for lack of reference we give a proof.
\begin{propgl}\label{prop.reduction_to_lambda=0}
Asume $p$ is good for $G$, let $P$ be a parabolic of $G$, let $\lambda\in X(P)$ be dominant, let $Q$ be the parabolic of $G$ containing $P$ such that $\lambda\in X(Q)$ and $\mc L_{G/Q}(\lambda)$ is ample, and let $L$ be the Levi subgroup of $Q$ containing $T$.
If $H^i(L/L\cap P,S(\fl/\fl\cap\p))=0$ for all $i>0$, then $H^i(G/P,S(\g/\p)\ot k_{-\lambda})$ for all $i>0$.
\end{propgl}
\begin{proof}
By \cite[Cor~3 to Thm~4]{KLT} or \cite[Thm~5.3]{BriKu} $G\times^P\fu_P$ is Frobenius split, so by \cite[Lemma~1.2.7(i)]{BriKu} it is enough to show the vanishing for $m\lambda$, $m\gg0$ (in fact we only need it for $p^m\lambda$ and some $m\ge0$).

Some of the arguments below are adaptations of arguments from the proof of \cite[Lem~5.2.7]{BriKu}.

Each $S^j(\g/\p)$ has a filtration with sections $S^r(\q/\p)\ot S^s(\g/\q)$, $r+s=j$, so it is enough to show that $R^i\ind_P^G(S(\q/\p)\ot S(\g/\q)\ot k_{-\lambda})=0$ for all $i>0$.
We have $P=(L\cap P)U_Q$ and $\q/\p\simeq \fl/\fl\cap\p$. Note that $U_Q$ acts trivially on $\q/\p$.
Combining \cite[I.6.11]{Jan} and our assumption with a standard spectral sequence argument, we have
\begin{align}\label{eq.collapsed_SS}
R^i\ind_P^G(S(\q/\p)\ot S(\g/\q)\ot k_{-\lambda})&\simeq R^i\ind_Q^G(\ind_P^QS(\q/\p)\ot S(\g/\q)\ot k_{-\lambda})\,.
\end{align}
We have $\p=\fl\cap\p\oplus\fu_Q$, $\fu_P=\fu_{L\cap P}\oplus\fu_Q$, $(\g/\p)^*\simeq\fu_P$, $(\g/\q)^*\simeq\fu_Q$, and $(\q/\p)^*\simeq\fu_{L\cap P}$.
By the arguments of \cite[p94]{Jan1} there exists an affine $Q$-variety $V_0$ such that $k[V_0]\simeq k[Q\times^P\fu_{L\cap P}]=\ind_P^QS(\q/\p)$, $Q$-equivariantly ($U_Q$ acting trivially).
Put $V=V_0\times\fu_Q$. Then $k[V]=\ind_P^QS(\q/\p)\ot S(\g/\q)$. Now the morphism $G\times^QV\to G/Q$ is affine, so by \cite[Ex~III.8.2]{H} the RHS of \eqref{eq.collapsed_SS} is isomorphic to \begin{align}\label{eq.Gx^QV}
H^i(G\times^QV,\mc L(\lambda))\,.
\end{align}
By \cite[5.1.12]{EGAII} $\mc L(\lambda)$ is ample on $G\times^QV$, since $G\times^QV\to G/Q$ is affine.
The morphism $V_0\to\ov{Q\cdot\fu_{L\cap P}}$ is finite, see \cite[p94]{Jan1}, so the same is true for the morphisms $V\to\ov{Q\cdot\fu_{L\cap P}}\times\fu$ and $G\times^QV\to G\times^Q(\ov{Q\cdot\fu_{L\cap P}}\times\fu_Q)$. So composing the latter with the projective morphism $G\times^Q(\ov{Q\cdot\fu_{L\cap P}}\times\fu_Q)\to\g$, given by the embedding of $\ov{Q\cdot\fu_{L\cap P}}\times\fu_Q$ in $\g$, we obtain a proper morphism $G\times^QV\to\g$.
Now \cite[III.5.3]{H} tells us that \eqref{eq.Gx^QV} is $0$ if we replace $\lambda$ by $m\lambda$, $m\gg0$.
\end{proof}

\section{The main results}\label{s.main}
Throughout this section, except in Theorem~\ref{thm.cohomology} and its proof, $G=\GL_n=\GL(k^n)$ and $T$ is the subgroup of diagonal matrices in $G$. As simple roots we choose the usual characters $\e_i-\e_{i+1}$, $1\le i\le n-1$, where we used additive notation for characters, and $\e_i$ is the $i$-th coordinate function on $T$. Then $B$ consists of the upper triangular matrices in $G$.
As is well-known, the conjugacy classes of parabolic subgroups of $G$ are labelled by the compositions of $n$, see e.g.\cite[3.2]{He}. By $\nu$ we denote a composition of $n$ and $P=P_\nu\supseteq B$ is the standard parabolic whose block sizes are given in order by $\nu$. If $A_\nu$ is the set $\{\nu_1,\nu_1+\nu_2,\ldots,\sum_{j=1}^{s-1}\nu_j\}$, $s$ the length of $\nu$, then $P_\nu=P_{I_\nu}$, the parabolic associated to the set of simple roots $I_\nu=\{\e_i-\e_{i+1}\,|\,i\in \{1,\ldots,n-1\}\sm A_\nu\}$. We denote by $\lambda$ the transposed partition of the weakly descending sorted version of $\nu$. It is well-kown that the Richardson orbit of $P_\nu$ is $\mc O_\lambda$, the nilpotent orbit whose Jordan block sizes are given by $\lambda$, see e.g. \cite[Thm 3.3(a)]{He}. 

It is well-known that the map $\varphi:[g,X]\mapsto gXg^{-1}:G\times^P\fu_P\to\ov{\mc O_\lambda}$ is birational. Indeed the group centraliser $G_X$ of any $X\in\g$ is the set of invertible elements in the Lie algebra centraliser $\g_X$, so is connected. Now see \cite[4.9 and 8.8 Remark]{Jan1}. It is also well-known that $\ov{\mc O_\lambda}$ is normal, see e.g. \cite{Don} or \cite[Sect~4.7]{MVdK}.

For $i\in\{1,\ldots,n-1\}$ we denote by $d_{\lambda,i}$ the number of nonzero positions on the $(n-i)$-th upper codiagonal of $\fu_P$.
So for $\nu=(2,1,2)$ we have $d_{\lambda,1},d_{\lambda,2},d_{\lambda,3},d_{\lambda,4}=1,2,3,2$, see the figure of $\fu_p$ below. $$\left[\begin{smallmatrix}0&0&*&*&*\\0&0&*&*&*\\0&0&0&*&*\\0&0&0&0&0\\0&0&0&0&0\end{smallmatrix}\right]$$
Since each diagonal $j\times j$ block of $P$ takes away $j-i$ nonzero positions from the $i$-th upper codiagonal, we have, if $j$ occurs $m_j$ times in $\nu$,
$$d_{\lambda,n-i}=n-i-\sum _{j>i} (j-i) m_j=n-i-\sum _{j>i}\lambda_j=-\sum _{j=i+1}^n (\lambda_j-1)=\sum _{j=1}^i(\lambda_j-1)\,.$$
Therefore, $d_{\lambda,i}=i-\sum _{j>n-i}\lambda_j=\sum _{j=1}^{n-i}(\lambda_j-1)$.
So indeed the $d_{\lambda,i}$ only depend on $\lambda$, moreover, they determine $\lambda$.

For a square matrix $X$ we denote by $X_{\le i,\le i}$ the submatrix of $X$ given by the first $i$ rows and columns. For an $i\times i$ matrix $Y$ we denote by $s_j(Y)$ the trace of the $j$-th exterior power of $Y$, i.e. the sum of the diagonal $j\times j$ minors of $Y$. As is well-known, $\det(aI_i-Y)=a^i+\sum_{j=1}^i(-1)^ja^{i-j}s_j(Y)$, where $I_i$ is the $i\times i$ identity matrix. So the largest $j$ with $s_j(Y)\ne0$ is the number of nonzero eigenvalues of $Y$, counted with (algebraic) multiplicity. This number also equals the rank of $Y^l$ for $l$ sufficiently big. We will call it the \emph{stable rank} of $Y$.

\begin{lemgl}\label{lem.upperbound}
Let $X\in\mc O_\lambda$.
\begin{enumerate}[{\rm(i)}]
\item Any $i$-dimensional subspace $W$ of $V=k^n$ contains an $X$-invariant subspace $U$ of dimension $\ge\sum_{j>n-i}\lambda_j$.
\item $X_{\le i,\le i}$ has stable rank $\le d_{\lambda,i}$.
\end{enumerate}
\end{lemgl}
\begin{proof}
(i).\ We show this by induction on $n$. It is trivial when $i\le n-r$, $r$ the length of $\lambda$, in particular when $n=0$. Assume $i>n-r$. Then $W$ has nonzero intersection with $\Ker(X)$ for dimension reasons. Pick $v$ nonzero in that intersection. First note that the transformation $\ov X$ induced on $V/kv$ by $X$ has partition $\mu$ which is obtained from $\lambda$ by subtracting 1 from one part of $\lambda$ and then sorting the result in weakly descending order. Indeed if we decompose $V$ as a direct sum of $X$-Jordan blocks and we pick a $X$-Jordan block of minimal size with the property that $v$ has nonzero component in it, then we can replace that $X$-Jordan block by an $X$-Jordan block of the same size which contains $v$. 
Now we apply the induction hypothesis to $V/kv$ and $W/kv$, noting that $(n-1)-(i-1)=n-i$, to obtain an $\ov X$-invariant subspace $U/kv$ of $W/kv$ of dimension $\ge\sum_{j>n-i}\mu_j\ge\sum_{j>n-i}\lambda_j-1$. Now $U$ is the $X$-invariant subspace we want.\\
(ii).\ The linear map $(X_{\le i,\le i})^i$ coincides with $X^i$ on any $X$-invariant subspace $U$ of $k^i\le k^n$ and therefore kills it. Choosing $U$ as in (i), it induces a linear map $k^i/U \to k^i$ and therefore has rank $\le i - \sum_{j>n-i}\lambda_j = d_{\lambda,i}$.
\end{proof}

Lemma~\ref{lem.existence} below follows from Lemma~\ref{lem.upperbound}(ii) and the existence of the KLT-splitting, but we prefer to give a direct proof.

\begin{lemgl}\label{lem.rank}
For any $h\in\{1,\ldots,i-1\}$ there exists a regular nilpotent $i\times i$ matrix $X$ such that $X_{\le h,\le h}$ is invertible.
\end{lemgl}
\begin{proof}
Let $(e_1,\ldots,e_i)$ be the standard basis of $k^i$. Then the regular nilpotent matrix $X$ given by $X(e_j)=e_{j-1}$ for $2<j\le i$, $X(e_2)=e_1+e_{h+1}$ and $X(e_1+e_{h+1})=0$ has the desired property. 
\end{proof}

\begin{rem}
Of course it follows from Lemma~\ref{lem.rank} that there exsists a regular $i\times i$ matrix $X$ such that $X_{\le h,\le h}$ is invertible for all $h\in\{1,\ldots,i-1\}$, but we won't need this.
\end{rem}

\begin{lemgl}\label{lem.existence}
There exists $X\in\mc O_\lambda$ such that $X_{\le i,\le i}$ has stable rank $d_{\lambda,i}$.
\end{lemgl}
\begin{proof}
First choose any $Y\in\g$ nilpotent with partition $\lambda$ and decompose $k^n$ into $Y$-Jordan blocks with sizes $\lambda_1, \lambda_2,\ldots,\lambda_r$, where $r$ is the length of $\lambda$. It suffices to find an ordered basis $\mc B$ of $k^n$ such that the upper left $i\times i$-block $Z$ of the matrix of $Y$ relative to this basis has stable rank $d_{\lambda,i}$.

Determine $s\le r$ maximal with $\sum_{j=1}^s(\lambda_j-1)\le i$ and put $h=i-\sum_{j=1}^s(\lambda_j-1)$. Using Lemma~\ref{lem.rank} choose for each $j\le s$ a basis of the $j$-th block such that the upper left $(\lambda_j-1)\times(\lambda_j-1)$ block of the matrix of $Y$ relative to this basis is invertible, if $s<r$ and $h>0$ choose a basis of the $(s+1)$-th block such that the upper left $h\times h$ block of the matrix of $Y$ relative to this basis is invertible, and for the remaining blocks choose any basis.

We now form $\mc B$ as follows. First consider the case $i\le n-r$. For each $j\le s$ we pick the first $\lambda_j-1$ basis vectors from the $j$-th block, if $s<r$ we append the first $h$ basis vectors from the $(s+1)$-th block, and finally we append all remaining $n-i$ basis vectors. Now $Z$ is in block diagonal form with invertible diagonal block of sizes $\lambda_1-1,\ldots,\lambda_s-1,h$, where $h$ has to be omitted if $h=0$. Now consider the case $i>n-r$. For each $j\le n-i$ we pick the first $\lambda_j-1$ basis vectors from the $j$-th block, then we append the basis vectors from the next $r-(n-i)$ blocks, and finally we append all remaining $n-i$ basis vectors. Now $Z$ is in block diagonal form with diagonal block sizes $\lambda_1-1,\ldots,\lambda_{n-i}-1,\lambda_{n-i+1},\ldots,\lambda_r$ where the first $n-i$ blocks are invertible, and the others nilpotent. In both cases we obtain that $Z$ has stable rank $d_{\lambda,i}$ (when $i\le n-r$ we have $d_{\lambda,i}=i$).
\end{proof}

Below we will denote a function $X\mapsto E(X)$ on a closed subvariety of $\g$ just by the expression $E(X)$.

\begin{thmgl}\label{thm.splitting} 
The degree $(p-1)\dim(G/P)$ component of the KLT splitting of $G\times^P\fu_P$ is the top degree component and equals the $(p-1)$-th power of the pullback of $\prod_{i=1}^{n-1}s_{d_{\lambda,i}}(X_{\le i,\le i})\in k[\ov{\mc O_\lambda}]$ along the resolution $\varphi:G\times^P\fu_P\to\ov{\mc O_\lambda}$. This pullback vanishes on the exceptional locus of $\varphi$.
\end{thmgl}
\begin{proof}
The KLT splitting is the pullback along $\varphi$ of the function given by \eqref{eq.splitting}. Furthermore, we have $\det(I_i+Y)=\sum_{j=0}^is_j(Y)$ for any $i\times i$ matrix $Y$, and, of course, $s_j(X_{\le i,\le i})\ne0$ on $\mc O_\lambda \iff s_j(X_{\le i,\le i})\ne0$ on $\ov{\mc O_\lambda}$.
So by Lemma's~\ref{lem.upperbound}(ii)~and~\ref{lem.existence} the top degree component of the $i$-th factor in \eqref{eq.splitting} is $s_{d_{\lambda,i}}(X_{\le i,\le i})$.
So the KLT-splitting has top degree $p-1$ times $\sum_{i=1}^{n-1}d_{\lambda,i}=\dim\fu_P=\dim(G/P)$, and the top degree component is the $(p-1)$-th power of the pullback along $\varphi$ of the function given by the stated formula.

To prove he second assertion, put $f_{\lambda,i}(X)=s_{d_{\lambda,i}}(X_{\le i,\le i})$ and $f_\lambda=\prod_{i=1}^{n-1}f_{\lambda,i}$. The exceptional locus is $\varphi^{-1}(\ov{\mc O_\lambda}\setminus\mc O_\lambda)$, so it suffices to show that $f_\lambda$ vanishes on any $\mc O_\mu\subseteq\ov{\mc O_\lambda}\setminus\mc O_\lambda$. We have $\dim(\fu_Q)=\frac{1}{2}\dim(\mc O_\mu)<\frac{1}{2}\dim(\mc O_\lambda)=\dim(\fu_P)$, where $Q$ is a standard parabolic whose Richardson orbit is $\mc O_\mu$, see \cite[4.9]{Jan1}. So for some $i$ we must have $d_{\mu,i}<d_{\lambda,i}$ which means that $f_{\lambda,i}$ and therefore $f_\lambda$ vanishes on $\mc O_\mu$.
\end{proof}

\begin{thmgl}\label{thm.cohomology} 
Let $G$ be any reductive group for which $p$ is good, let $\lambda\in X(T)$ be dominant, put $I=\{\alpha\in S\,|\,\la\lambda,\alpha^\vee\ra=0\}$.
Then $H^i(T^\vee(G/P_J),\mc L(\lambda))=0$ for all $J\subseteq I$ such that $R_J$ contains all irreducible components of $R_I$ not of type $A$. 
\end{thmgl}
\begin{proof}
By Proposition~\ref{prop.reduction_to_lambda=0} we may assume that $\lambda=0$ and that all irreducible components of $R$ have type $A$. Since we are dealing with cotangent bundles we may assume that $G$ is semsimple and simply connected. By the K\"unneth formula \cite[Prop~9.2.4]{Kempf} we may then assume $G=\SL_n$ and finally we may assume $G=\GL_n$.
Now the result follows from Theorem~\ref{thm.splitting} and \cite[Thm~1.2]{MVdK0}, bearing in mind that the canonical bundle of $T^\vee(G/P)$ is trivial, see \cite[Lem~5.1.1]{BriKu}, and that $R^i\varphi_*(\mc O_{T^\vee(G/P)})$ is the sheaf associated with the cohomology group $H^i(T^\vee(G/P),\mc O_{T^\vee(G/P)})$, since $\varphi$ is affine.
\end{proof}

We remind the reader that a proper birational morphism $\psi:X\to Y$ is called a \emph{rational resolution} if $\psi_*\mc O_X=\mc O_Y$ and the higher direct images of $\mc O_X$ and $\omega_X$ are $0$, see \cite[Def~3.4.1]{BriKu}.
\begin{cornn} 
The resolution $\varphi:G\times^P\fu_P\to\ov{\mc O_\lambda}$ is rational. 
\end{cornn}
\begin{proof}
This follows from a standard argument, see e.g. \cite[Lem~14.5]{Jan}, and Theorem~\ref{thm.cohomology}.
\end{proof}

\begin{rems}
1.\ If $P=B$, then $d_{(n),i}=i$ for all $i$, so the splitting from Theorem~\ref{thm.splitting} equals the $(p-1)$-th power of the pullback of $\prod_{i=1}^{n-1}\det(X_{\le i,\le i})$ along\break $\varphi:G\times^B\fu\to\mc N$. This is the MVdK splitting of $G\times^B\fu$, see \cite{MVdK}.\\
2.\ Thomsen mentioned to me another proof of Lemma~\ref{lem.upperbound}(ii): One can easily deduce it from the following result which can be proved by induction on $n$. For $X\in\fu_P$ let $y_{ij}=\delta_{ij}+x_{ij}$ be the $(i,j)$-th entry of $I_n+X$. Then any monomial $y_{i_1j_1}y_{i_2j_2}\cdots y_{i_sj_s}$ with the $i_l$ all distinct and the $j_l$ all distinct has degree $\le d_{\lambda,s}$ in the $x_{ij}$.
\end{rems}

\begin{connn}[Thomsen]
The pushforward to $G\times^P\fu_P$ of the splitting of $G\times^B\fu_P$ induced by the MVdK splitting is the top degree component of the KLT splitting.
\end{connn}

\section*{Acknowledgement}
I would like to thank Jesper Funch Thomsen for helpful discussions.


\begin{thebibliography}{99}
\bibitem{BNPP} C.~Bendel, D.~Nakano, B.~Parshall, C.~Pillen, {\it Cohomology for quantum groups via the geometry of the nullcone}, Mem. Amer. Math. Soc.~{\bf229} (2014), no. 1077.
\bibitem{BriKu} M.~Brion, S.~Kumar, {\it Frobenius splitting methods in geometry and representation theory}, Progress in Mathematics, 231, Birkh\"auser Boston, Inc., Boston, MA, 2005.
\bibitem{Broer1} A.~Broer, {\it Line bundles on the cotangent bundle of the flag variety}, Invent. Math. {\bf 113} (1993), no. 1, 1-20.
\bibitem{Broer2} A.~Broer, {\it Normality of some nilpotent varieties and cohomology of line bundles on the cotangent bundle of the flag variety}, in ``Lie Theory and Geometry", Prog. Math. {\bf123} (1994), Birkh\"auser, Boston, 1-19.
\bibitem{Don} S.~Donkin, {\it The normality of closures of conjugacy classes of matrices}, Invent. Math. {\bf 101}  (1990), no. 3, 717-736.
\bibitem{EGAII} A.~Grothendieck, {\it El\'ements de g\'eom\'etrie alg\'ebrique II, \'Etude globale \'el\'ementaire de quelques classes de morphismes}, Inst. Hautes \'Etudes Sci. Publ. Math. No. 8, 1961.
\bibitem{H} R.~Hartshorne, {\it Algebraic Geometry}, GTM {\bf 52}, Springer-Verlag, New York-Heidelberg, 1977.
\bibitem{He} W.~H.~Hesselink, {\it Polarizations in the classical groups}, Math. Z. {\bf160} (1978), no. 3, 217-234.
\bibitem{Jan} J.~C.~Jantzen, {\it Representations of algebraic groups}, Second edition, American Mathematical Society, Providence, RI, 2003.
\bibitem{Jan1} J.~C.~Jantzen, {\it Nilpotent orbits in representation theory}, ``Lie theory", Progr. Math. {\bf228}, Birkh\"auser Boston, 2004, 1-211.
\bibitem{Kempf} G.~R.~Kempf, {\it Algebraic Varieties}, London Math. Soc. Lecture Note Ser. 172, Cambridge University Press, Cambridge, 1993. 
\bibitem{KLT} S.~Kumar, N.~Lauritzen,J.~F.~Thomsen, {\it Frobenius splitting of cotangent bundles of flag varieties}, Invent. Math. {\bf 136} (1999), no.3, 603-621.
\bibitem{MVdK0} V.~B.~Mehta, W.~van der Kallen, {\it On a Grauert-Riemenschneider vanishing theorem for Frobenius split varieties in characteristic $p$}, Invent. Math. {\bf108} (1992), no. 1, 11-13.
\bibitem{MVdK} V.~B.~Mehta, W.~van der Kallen, {\it A simultaneous Frobenius splitting for closures of conjugacy classes of nilpotent matrices}, Compositio Math. {\bf84} (1992), no. 2, 211-221.
\bibitem{Th} J.~F.~Thomsen, {\it Normality of certain nilpotent varieties in positive characteristic}, J. Algebra {\bf227} (2000), no. 2, 595-613.
\end{thebibliography}
\end{document}